\RequirePackage{fix-cm}
\documentclass[smallextended]{svjour3}       % onecolumn (second format)
\smartqed  % flush right qed marks, e.g. at end of proof
\usepackage{graphicx}
\usepackage{multirow}
\usepackage{amsmath,amssymb,amsfonts}
\usepackage{mathrsfs}
\usepackage[title]{appendix}
\usepackage{xcolor}
\usepackage{textcomp}
\usepackage{manyfoot}
\usepackage{booktabs}
\usepackage{algorithm}
\usepackage{algorithmicx}
\usepackage{algpseudocode}
\usepackage{listings}
\usepackage{bm}
\usepackage{mathtools}
\usepackage{etoolbox}

\newcommand\argmin{\mathop{\rm argmin}}

\newcommand\dist{\mathop{\rm dist}}

\newcommand\prox{\mathop{\rm prox}}
\newcommand\proj{\mathrm{proj}}
\newcommand\conv{\mathop{\rm conv}}

\newcommand\bz{\bm{0}}

\newcommand{\blambda}{\bm{\lambda}}
\newcommand{\bmu}{\bm{\mu}}

\renewcommand{\a}{\bm{a}}
\renewcommand{\b}{\bm{b}}
\renewcommand{\c}{\bm{c}}
\renewcommand{\d}{\bm{d}}
\newcommand{\R}{{\mathbb{R}}}
\newcommand\p{\bm{p}}
\newcommand\q{\bm{q}}
\renewcommand\r{\bm{r}}
\newcommand\n{\bm{n}}

\renewcommand\u{\bm{u}}
\renewcommand\v{\bm{v}}
\newcommand\w{\bm{w}}
\newcommand\x{\bm{x}}
\newcommand\y{\bm{y}}
\newcommand\z{\bm{z}}

\newcommand\e{\bm{e}}
\newcommand\mA{\mathcal{A}}
\newcommand\mC{\mathcal{C}}

\newcommand\mN{\mathcal{N}}

\renewcommand\O{\mathcal{O}}

\newcommand\Id{\mathop{\mathrm{Id}}}
\newcommand\cl{\mathop{\mathrm{cl}}}
\newcommand\range{\mathop{\mathrm{range}}}

\AtEndEnvironment{proof}{\phantom{}\qed}

\renewcommand{\makeheadbox}{}
\makeatletter
\patchcmd{\@maketitle}
  {\hrule\@height0.35mm\noindent}
  {}
  {}
  {\PackageWarning{esp}{Could not remove title rule}}
\makeatother

\begin{document}

% \title{Ellipsoid Separation in Data-Dependent Time via Infimal Displacement Vectors\thanks{This work was supported in part by a Discovery Grant from
% the Natural Sciences and Engineering Research Council
% of Canada (NSERC). In the event that this manuscript is accepted, the authors reserve the right to post the final accepted manuscript on arxiv.org without payment of a fee and notwithstanding any clause in the publisher's contract in order to comply with NSERC's open access policy.}}
\title{Ellipsoid Separation in Data-Dependent Time via Infimal Displacement Vectors\thanks{This work was supported in part by a Discovery Grant from
the Natural Sciences and Engineering Research Council
of Canada (NSERC).}}

\titlerunning{Ellipsoid Separation in Data-Dependent Time via Infimal Displacement Vectors}        % if too long for running head

\author{Matthew Hough         \and
        Stephen A. Vavasis
}


\institute{Matthew Hough \and Stephen A. Vavasis \at
    Department of Combinatorics \& Optimization,
    University of Waterloo,
    Waterloo, ON N2L 3G1, Canada\\
    \email{mhough@uwaterloo.ca}\\
    \email{vavasis@uwaterloo.ca}
}

\date{}

\maketitle

\begin{abstract}
The infimal displacement vector of the nonexpansive operator underlying the primal-dual hybrid gradient algorithm (PDHG) is known to encode useful information for infeasible optimization problems. In particular, for the ellipsoid separation problem (ESP), where the goal is to find a separating hyperplane between two collections of ellipsoids, the infimal displacement vector corresponding to an infeasible second-order cone program (SOCP) encodes a separating hyperplane. We show FISTA computes a sequence of approximations to the dual component of the infimal displacement vector and establish that a solution to the ESP can be obtained in $\O(\delta_{\mA}^{-2})$ iterations. 
We then show how to improve the bound to $\O(\delta_{\mA}^{-1})$.
Here $\delta_{\mA}$ denotes the distance between the convex hulls of the two ellipsoid collections.
\end{abstract}

\section{Introduction}
Recent work has shown that when an optimization problem is infeasible or unbounded, the iterates of the primal-dual hybrid gradient algorithm (PDHG)
can encode useful information through the \emph{infimal displacement vector} (IDV) \cite{Applegate2024,Jiang2023}. In particular, in \cite{Applegate2024} the authors show that the IDV is able to certify infeasibility and unboundedness of linear programs.
Jiang et al. \cite{Jiang2023} further show that the IDV associated with an infeasible second-order cone program (SOCP) can be used to solve the ellipsoid separation problem (ESP), which seeks a hyperplane that strictly separates two finite collections of ellipsoids. This problem arises in robust classification, where each data point is represented by an ellipsoid describing uncertainty in its location \cite{Shivaswamy2006}.

Kalantari's Triangle Algorithm \cite{Kalantari2015,Kalantari2019} can be used to separate disjoint compact convex sets.
Applied to the convex hulls of the two ellipsoid collections, it yields an $\O(\delta_{\mA}^{-2})$ iteration complexity under our normalization assumptions,
where $\delta_{\mA}$ denotes the distance between the two convex hulls.

In the infeasible SOCP formulation considered in \cite[Section~6.2]{Jiang2023}, differences between successive PDHG iterates converge to the IDV, and thus solve the ESP asymptotically. In this paper, we reformulate the SOCP solved by PDHG as a least-squares problem with second-order cone constraints
so that it can be solved by FISTA, and show that the FISTA residuals converge to the dual component of the IDV which encodes a separating hyperplane for the ESP.
We establish that a separating hyperplane can be obtained in finitely many iterations and derive a data-dependent iteration complexity for obtaining such a hyperplane. In particular, we first establish an $\O(\delta_{\mA}^{-2})$ iteration complexity bound,
which we subsequently improve to $\O(\delta^{-1}_{\mA})$ by considering a weighted average of the FISTA residuals.
Thus, the number of iterations required depends on the geometry of the problem, with better-separated ellipsoids admitting faster convergence to a separator, highlighting a novel role for the IDV.

\section{Primal-dual hybrid gradient}
The primal-dual hybrid gradient algorithm (PDHG) is an algorithm for solving saddle point problems of the form
\begin{equation} \label{prob:sp}
    \inf_{\x \in \R^n} \sup_{\y \in \R^m} \y^TA\x + f(\x) - g(\y)
\end{equation}
where $A : \R^n \to \R^m$ is linear and $f$ and $g$ are proper, lower semicontinuous, convex functions. Let $(\x_0,\y_0) \in \R^n \times \R^m$ and suppose the parameters $\sigma,\tau > 0$ are chosen to satisfy $\sigma\tau\|A\|^2 < 1$. PDHG generates its iterates by repeatedly applying the operator $T$:
\[
    (\x_{k+1}, \y_{k+1}) = T(\x_k,\y_k) = \begin{pmatrix}
        \prox_{\sigma f}(\x_k - \sigma A^T\y_k)\\
        \prox_{\tau g}(\y_k + \tau A(2\prox_{\sigma f}(\x_k - \sigma A^T\y_k) - \x_k))
    \end{pmatrix}.
\]
Under appropriate constraint qualifications, $(\x^*,\y^*)$ is a fixed point of $T$ if and only if $(\x^*,\y^*)$ is optimal for \eqref{prob:sp} \cite[Proposition~3.1]{Jiang2023}. In the inconsistent case, where $\bz \notin \range (\Id - T)$, the iterates $(\x_k,\y_k)$ diverge \cite{Reich1,Reich2}. However, the successive differences $(\x_{k},\y_{k}) - (\x_{k+1},\y_{k+1})$ converge to the IDV $\v = (\v_R, \v_D)$, defined as
\[
    \v := \argmin_{\w \in \cl\left(\range(\Id - T)\right)} \|\w\|_M,
\]
where $\|\w\|_M = \sqrt{\w^TM\w}$ is the norm induced by the matrix
\[
    M = \begin{pmatrix}
        \sigma^{-1}I & -A^T\\
        -A & \tau^{-1}I
    \end{pmatrix}.
\]
The assumption $\sigma\tau\|A\|^2 < 1$ ensures that $M$ is positive definite.

\section{Ellipsoid separation as an SOCP} \label{sec:esp}
Given two finite collections of ellipsoids in $\R^d$,
\[
    \{C_1,\ldots,C_j\}, \qquad \{D_1,\ldots,D_l\},
\]
the ellipsoid separation problem asks whether there exists a hyperplane strictly separating the two collections. We write the ellipsoids as
\[
\begin{aligned}
    C_i &= \{\z:\|A_i^{-1}(\z-\c_i)\|\leq 1\}, &&i=1,\ldots,j,\\
    D_i &= \{\z:\|B_i^{-1}(\z-\d_i)\|\leq 1\}, &&i=1,\ldots,l,
\end{aligned}
\]
where $A_i, B_i \in \R^{d\times d}$ are nonsingular and $\c_i,\d_i\in\R^d$ are the ellipsoid centers. We make the assumption throughout that the data is scaled so that
\begin{equation} \label{ass} \tag{A}
    \max\left\{\max_i\{\|\c_i\|,\|A_i\|\}, \max_i\{\|\d_i\|,\|B_i\|\}\right\}=1.
\end{equation}
We seek $\w\in\R^d$ and $s,t\in\R$ such that
\[
\begin{aligned}
    \w^T\z &\leq -s, &&\forall \z\in C_i,\quad i=1,\ldots,j,\\
    \w^T\z &\geq t, &&\forall \z\in D_i,\quad i=1,\ldots,l,
\end{aligned}
\]
with $s+t>0$. In this case, any hyperplane $\{\z:\w^T\z=h\}$ with $h\in(-s,t)$ strictly separates the two collections. By \cite[Lemma~6.18]{Jiang2023}, these halfspace containment conditions are equivalent to
\[
    \begin{pmatrix}
        s+\c_i^T\w\\ A_i^T\w
    \end{pmatrix} \in -C_2^{d+1}, \qquad
    \begin{pmatrix}
        t-\d_i^T\w\\ B_i^T\w
    \end{pmatrix} \in -C_2^{d+1},
\]
for all $i=1,\ldots,j$ and $i=1,\ldots,l$, respectively, where $C_2^{d+1}$ denotes the second-order cone. Consequently, the ellipsoid separation problem can be formulated as the following homogeneous second-order cone program:
\begin{equation}\label{prob:socp-esp-primal}
\begin{array}{rl}
  \max_{s,t,\w} & s+t \\
    \mbox{s.t.} & \mA^T\begin{pmatrix} s \\ t \\ \w\end{pmatrix}
    \preceq_{\mC} \bz,
\end{array}
\end{equation}
where the matrix $\mA^T\in\R^{ (j+l)(d+1)\times(d+2)}$ is
  \[
    \mA^T := \begin{pmatrix}
      \bar{A}_1^T\\
      \vdots\\[4pt]
      \bar{A}_j^T\\[4pt]
      \bar{B}_1^T\\
      \vdots\\[4pt]
      \bar{B}_l^T
    \end{pmatrix},
  \]
  with
  \[
    \bar{A}_i^T:=
    \begin{pmatrix}
    1 & 0 & \c_i^T \\
    \bz & \bz & A_i^T
    \end{pmatrix},\quad
    \bar{B}_i^T:=
    \begin{pmatrix}
        0 & 1 & -\d_i^T \\
        \bz & \bz & B_i^T
    \end{pmatrix},
  \]
and the notation $\a \preceq_{\mC} \b$ means that $\a - \b \in -\mC$, where
  \[
    \mC:=\underbrace{C_2^{d+1}\times\cdots\times C_2^{d+1}}_{(j+l) \ \mathrm{times}}.
  \]
Due to homogeneity, the objective does not yield a finite maximum; rather, feasibility with $s + t > 0$ certifies that the halfspaces are disjoint and a separator exists. Otherwise, the optimum is $0$, certifying that no separating hyperplane exists.

\section{FISTA for ellipsoid separation} \label{sec:esp-fista}
It is not apparent how to apply FISTA \cite{FISTA,Nesterov2004} directly to \eqref{prob:socp-esp-primal}, since there is no closed-form expression to project
onto the constraint set $\{(s,t,\w) : \mA^T(s,t,\w) \preceq_{\mC} \bz\}$. It turns out that the dual SOCP admits a constraint set that can be efficiently
projected onto. Introduce the dual variables $\lambda_1,\hdots,\lambda_j,\mu_1,\hdots,\mu_l \in \R$ and $\p_1,\hdots,\p_j,\q_1,\hdots,\q_l \in \R^d$, and define
\[
  \x := \begin{pmatrix}
    \lambda_1\\
    \p_1\\
    \vdots\\
    \lambda_j\\
    \p_j\\
    \mu_1\\
    \q_1\\
    \vdots\\
    \mu_l\\
    \q_l
  \end{pmatrix}.
\]
The dual SOCP is then
\begin{equation}\label{prob:socp-esp-dual}
\begin{array}{rl}
  \min_{\x} & 0\\
\mbox{s.t.} & \mA\x = \b\\
&\x \succeq_{\mC} \bz,
\end{array}
\end{equation}
where $\b = (1,1,\bz)$. Written explicitly, the constraint $\mA\x = \b$ becomes
\begin{equation} \label{eqn:esp-explicit-constraint}
  \begin{split}
    &\lambda_1+\cdots+\lambda_j=1,\\
    &\mu_1+\cdots+\mu_l=1, \\
    &\sum_{i=1}^j\left(\lambda_i\c_i+A_i\p_i\right)-\sum_{i=1}^l\left(\mu_i\d_i-B_i\q_i\right)=\bz,
  \end{split}
\end{equation}
while the constraint $\x \succeq_{\mC} \bz$ becomes
\begin{align*}
  \Vert \p_i\Vert&\le \lambda_i &&\forall i=1,\ldots, j \\
  \Vert \q_i\Vert&\le \mu_i &&\forall i=1,\ldots, l.
\end{align*}
The dual feasibility problem \eqref{prob:socp-esp-dual} has a simple geometric interpretation. The constraint $\mA\x=\b$, together with $\x\in\mC$, is feasible if and only if
\[
    \conv\left(C_1 \cup C_2 \cup \cdots \cup C_j\right) \cap \conv\left(D_1\cup D_2 \cup \cdots \cup D_l\right) \neq \emptyset.
\]
Consequently, the ellipsoid separation problem has a solution if and only if the dual conic feasibility problem is infeasible. To obtain a separating hyperplane, we apply FISTA to the problem
\begin{equation}\label{prob:socp-esp-dual2}
\begin{array}{rl}
  \min_{\x} & f(\x):=\frac{1}{2}\|\mA\x-\b\|^2\\
\mbox{s.t.} & \x \succeq_{\mC} \bz.
\end{array}
\end{equation}
This formulation is amenable to FISTA because $\nabla f$ is Lipschitz continuous with constant $L = \|\mA\|^2$ and projection onto the product of second-order cones $\mC$ can be computed in linear time (cf. \cite[Example 6.37]{Beck2017}).

Let $\x_k$ denote the iterates generated by FISTA and define the residuals
\[
\v_k := \b-\mA\x_k.
\]
We next establish how these residuals can be used to construct a separating hyperplane, and subsequently bound the number of iterations required to obtain one.

In the following analysis, we consider the infimal displacement vector of PDHG applied to \eqref{prob:socp-esp-dual} with $\tau = 1$ and $\sigma > 0$ satisfying $\sigma\|\mA\|^2 < 1$.

\begin{theorem} \label{thm:pdhg-to-fista}
    For any minimizer $\bar{\x}$ of \eqref{prob:socp-esp-dual2}, the dual component of the PDHG infimal displacement vector, $\v_D$, satisfies
    \[
        \v_D = \b - \mA\bar{\x}.
    \]
    Furthermore, the FISTA residuals converge to the dual part of the IDV: $\v_k \to \v_D$.
\end{theorem}
\begin{proof}
  Jiang et al. note that $\ker(\mA) \cap \mC = \{\bz\}$ \cite[Section~6.2]{Jiang2023}, so we may take $\tau = 1$ in \cite[Lemma~6.8 (vi)]{Jiang2023} to obtain  
  \begin{equation} \label{eqn:vd-proj}
    \v_D = \proj_{\b - \mA(\mC)}(\bz) = \proj_{(\mA(\mC))^{\circ}}(\b).
  \end{equation}
  In particular, a minimizer $\bar{\x}$ of \eqref{prob:socp-esp-dual2} exists, and every such minimizer satisfies $\v_D=\b-\mA\bar{\x}$. For any $\z \in \mA(\mC)$, the first-order optimality conditions of the projection give
  \begin{equation} \label{eqn:proj-ineq}
    \v_D^T\left(\v_D - (\b - \z)\right) \leq 0.
  \end{equation}
  We may write
  \[
    \lVert\mA\x - \b\rVert^2 - \lVert\mA\bar{\x} - \b\rVert^2 = \lVert\mA\x - \mA\bar{\x}\rVert^2 + 2(\mA\bar{\x} - \b)^T(\mA\x - \mA\bar{\x}),
  \]
  and substitute $\b - \mA\bar{\x} = \v_D$ to obtain
  \[
    \lVert\mA\x - \b\rVert^2 - \lVert\mA\bar{\x} - \b\rVert^2 = \lVert\mA\x - \mA\bar{\x}\rVert^2 - 2\v_D^T(\v_D - (\b - \mA\x)).
  \]
  Let $\v := \b - \mA\x$. Applying \eqref{eqn:proj-ineq} gives the inequality:
  \begin{align*}
    \lVert\mA\x - \b\rVert^2 - \lVert\mA\bar{\x} - \b\rVert^2 &= \lVert\mA\x - \mA\bar{\x}\rVert^2 - 2\v_D^T(\v_D - (\b - \mA\x))\\
    &\geq \lVert\mA\x - \mA\bar{\x}\rVert^2\\
    &= \lVert\v - \v_D\rVert^2.
  \end{align*}
  Taking $\x = \x_k$ where $\x_k$ represents the $k$th FISTA iterate, we may use the fact that the iterates of FISTA converge in objective value, so $\v_k \to \v_D$.
\end{proof}
The preceding theorem establishes that the FISTA residuals converge to the dual component of the PDHG infimal displacement vector. The significance of this result is that $\v_D$ encodes a separating hyperplane whenever the two collections of ellipsoids are strictly separable. We show this in the following lemma.

\begin{lemma} \label{lem:vd-separates}
  Suppose the two collections of ellipsoids are strictly separable, and denote
  \[
    \v_D = \begin{pmatrix}
        s^*\\
        t^*\\
        \w^*
    \end{pmatrix}.
  \]
  Then $s^* + t^* = \|\v_D\|^2 > 0$, and
  \begin{align*}
      &(\w^*)^T\z \leq -s^*,\qquad\forall \z \in C_1 \cup \hdots \cup C_j\\
      &(\w^*)^T\z \geq t^*,\qquad\forall \z \in D_1 \cup \hdots \cup D_l.
  \end{align*}
  That is, any hyperplane $\{\z : (\w^*)^T\z = h\}$ with $h \in (-s^*,t^*)$ strictly separates the two collections.
\end{lemma}
\begin{proof}
    Taking $\tau = 1$ in \cite[Lemma~6.8 (viii)]{Jiang2023} gives $\|\v_D\|^2 = \b^T\v_D$. But $\b = (1,1,\bz)$, so
    \[
        \|\v_D\|^2 = \b^T\v_D = s^* + t^*.
    \]
    By \eqref{eqn:vd-proj}, $\v_D \in (\mA(\mC))^{\circ}$, thus $\mA^T\v_D \in \mC^{\circ}$. Therefore $\v_D$ is feasible for \eqref{prob:socp-esp-primal} with $s^* + t^* > 0$. The constraints of \eqref{prob:socp-esp-primal} encode
    \[
        \begin{pmatrix}
            s+\c_i^T\w\\ A_i^T\w
        \end{pmatrix} \in -C_2^{d+1}, \qquad
        \begin{pmatrix}
            t-\d_i^T\w\\ B_i^T\w
        \end{pmatrix} \in -C_2^{d+1},
    \]
    for all $i=1,\ldots,j$ and $i=1,\ldots,l$, respectively. The conclusion follows by the equivalence between the halfspace containment conditions and the second-order cone membership shown above \cite[Lemma~6.18]{Jiang2023}.
\end{proof}
We now have a procedure for obtaining a separating hyperplane from the optimal FISTA residual $\v_D = (s^*,t^*,\w^*)$. Simply choose any $h \in (-s^*, t^*)$ and the corresponding separating hyperplane is $\{\z : (\w^*)^T\z = h\}$. However, $\v_D$ is only guaranteed to be obtained from FISTA in the limit. Our goal is to find a separator after finitely many iterations, without computing $\v_D$ exactly. To this end, write
\[
    \v_k = \begin{pmatrix}
        s_k\\
        t_k\\
        \w_k
    \end{pmatrix},
\]
and consider the candidate separating hyperplane
\[
    H_k := \{\z : \w_k^T\z = m_k\},\qquad m_k := \frac{-s_k + t_k}{2}.
\]
The next theorem establishes how close $\v_k$ must be to $\v_D$ in order to guarantee that $H_k$ separates the two collections of ellipsoids.
\begin{theorem} \label{thm:vk-vd-approx-sep-bound}
  Suppose the two collections of ellipsoids are strictly separable. The residual $\v_k$ defines a strict separator $\{\z : \w_k^T\z = m_k\} = H_k$ provided
  \[
    \lVert\v_k - \v_D\rVert < \frac{\lVert\v_D\rVert^2}{\sqrt{18}}
  \]
\end{theorem}
\begin{proof}
  For any $\z \in C_1\cup\hdots\cup C_j$, we have that $(\w^*)^T\z \leq -s^*$. This inequality holds true if and only if
  \[
    \w_k^T\z \leq m_k + (\w_k - \w^*)^T\z - s^* - m_k.
  \]
  Let
  \[
    r := \lVert\w_k - \w^*\rVert\cdot\lVert\z\rVert - \frac{s^* - s_k}{2} + \frac{t^* - t_k}{2} - \frac{s^* + t^*}{2}.
  \]
  By Cauchy-Schwarz, any $\z \in C_1\cup\hdots\cup C_j$ satisfies
  \begin{align*}
    \w_k^T\z &= m_k + (\w_k - \w^*)^T\z + {\w^*}^T\z - m_k\\
             &\leq m_k + \lVert\w_k - \w^*\rVert\cdot\lVert\z\rVert - s^* - m_k\\
             &= m_k + r.
  \end{align*}
  It follows that if we have $r < 0$, then $\w_k^T\z < m_k$. Under the normalization assumption \eqref{ass},
  we have that $\lVert\z\rVert \leq 2$, and it is easy to verify that $r < 0$ is implied by the condition
  \begin{equation} \label{eqn:splust-condition}
    s^* + t^* > \lvert s^* - s_k\rvert + \lvert t^* - t_k\rvert + 4\lVert\w_k - \w^*\rVert.
  \end{equation}
  The implication $\z \in D_1 \cup \hdots \cup D_l \implies \w_k^T\z > m_k$ holds under the same assumption.
  Therefore, the satisfaction of the condition \eqref{eqn:splust-condition} implies that the residual $\v_k$ defines
  a strict separator.

  Now observe the bound
  \[
    \lvert s^* - s_k\rvert + \lvert t^* - t_k\rvert + 4\lVert\w_k - \w^*\rVert \leq \sqrt{18}\lVert\v_k - \v_D\rVert
  \]
  holds, so a new condition for $\v_k$ defining a strict separator is
  \[
    \lVert\v_k - \v_D\rVert < \frac{s^* + t^*}{\sqrt{18}}.
  \]
  From Lemma~\ref{lem:vd-separates}, $s^* + t^* = \|\v_D\|^2$, which completes the proof.
\end{proof}
We have shown that sufficiently accurate FISTA residuals yield separating hyperplanes. In the next section we derive a bound on the number of FISTA iterations required before $\v_k$ is sufficiently accurate per Theorem~\ref{thm:vk-vd-approx-sep-bound}.

\section{A data-dependent iteration complexity bound} \label{sec:esp-data-dep-bound}
In order to derive a data-dependent bound on the number of FISTA iterations needed to obtain a separating hyperplane, we relate $\lVert\v_D\rVert$ to the distance between the convex hulls of the two ellipsoid collections.

Our strategy is to use a minimizer of \eqref{prob:socp-esp-dual2} to construct a point in the convex hull of the $C_i$ ellipsoids and another in the convex hull of the $D_i$ ellipsoids, whose distance can be bounded in terms of $\|\v_D\|$. Combining this geometric bound with the convergence rate of FISTA provides an iteration complexity bound that depends on the separation between the two collections.

Let $\x^* \in \mC$ such that $\b - \mA\x^* = \v_D$, which exists by Theorem~\ref{thm:pdhg-to-fista}. Let
\[
  \x^* := \begin{pmatrix}
    \lambda_1^*\\
    \p_1^*\\
    \vdots\\
    \lambda^*_j\\
    \p^*_j\\
    \mu^*_1\\
    \q^*_1\\
    \vdots\\
    \mu^*_l\\
    \q^*_l
  \end{pmatrix}\quad\quad\text{and}\quad\quad \hat{\x} := \begin{pmatrix}
    \hat{\lambda}_1\\
    \hat{\p}_1\\
    \vdots\\
    \hat{\lambda}_j\\
    \hat{\p}_j\\
    \hat{\mu}_1\\
    \hat{\q}_1\\
    \vdots\\
    \hat{\mu}_l\\
    \hat{\q}_l
  \end{pmatrix},
\]
where $\blambda^* = (\lambda_1^*,\hdots,\lambda_j^*)$, $\hat{\blambda} = (\hat{\lambda}_1,\hdots,\hat{\lambda}_j)$, $\bmu^* = (\mu_1^*,\hdots,\mu_l^*)$, $\hat{\bmu} = (\hat{\mu}_1,\hdots,\hat{\mu}_l)$, and
\begin{align*}
  &\hat{\blambda} := \proj_{\Delta^j}(\blambda^*)\\
  &\hat{\bmu} := \proj_{\Delta^l}(\bmu^*)\\
  &\hat{\p}_i := \begin{cases}
    \frac{\hat{\lambda}_i}{\lambda_i^*}\p_i^*, & \hat{\lambda}_i < \lambda_i^*;\\
    \p_i^*, & \text{otherwise}
  \end{cases}\\
  &\hat{\q}_i := \begin{cases}
    \frac{\hat{\mu}_i}{\mu_i^*}\q_i^*, & \hat{\mu}_i < \mu_i^*\\
    \q_i^*, & \text{otherwise},
  \end{cases}
\end{align*}
where $\Delta^t = \{\blambda \in \R^t : \lambda_i \geq 0,\ \forall i \in [t],\quad \sum_{i=1}^t \lambda_i = 1\}$. By construction, $\hat{\x} \in \mC$.

\begin{lemma} \label{lem:simplex-proj}
  $\lVert\blambda^* - \hat{\blambda}\rVert_1 = \lvert 1 - \e^T\blambda^*\rvert = \lvert s^*\rvert$ and $\lVert\bmu^* - \hat{\bmu}\rVert_1 = \lvert 1 - \e^T\bmu^*\rvert = \lvert t^*\rvert$.
\end{lemma}
\begin{proof}
    We will prove $\lVert\blambda^* - \hat{\blambda}\rVert_1 = \lvert 1 - \e^T\blambda^*\rvert = \lvert s^*\rvert$ only, since the proof of the second expression is analogous.

    By definition of $\mA$ and the identity $(s^*,t^*,\w^*) = \v_D = \b - \mA\x^*$, we have $s^* = 1 - \e^T\blambda^*$.

    Recall $\hat{\blambda} := \proj_{\Delta^j}(\blambda^*)$ is given by
    \[
        \hat{\lambda}_i = \max\{\lambda_i^* - \eta, 0\},
    \]
    where $\eta \in \R$ is chosen to make $\sum_{i=1}^j \hat{\lambda}_i = 1$. Since $\lambda^*_i \geq \|\p_i^*\|$, we have $\lambda^*_i \geq 0$.
    
    If $\sum_{i=1}^j \lambda^*_i \geq 1$, then $\eta \geq 0$, so for all $i$
    \[
        \hat{\lambda}_i \leq \lambda_i^*
    \]
    On the other hand, if $\sum_{i=1}^j \lambda^*_i < 1$, then $\eta < 0$, so for all $i$
    \[
        \hat{\lambda}_i \geq \lambda_i^*.
    \]
    It follows that all $\lambda_i^* - \hat{\lambda}_i$ must have the same sign for all $i$. So we may write
    \begin{align*}
        \|\blambda^* - \hat{\blambda}\|_1 &= \sum_{i=1}^j|\lambda_i^* - \hat{\lambda}_i| = \left|\sum_{i=1}^j\left(\lambda_i^* - \hat{\lambda}_i\right)\right| = |1-\e^T\blambda^*|.
    \end{align*}
\end{proof}

\begin{lemma} \label{lem:pq-diff}
  $\lVert\hat{\p}_i - \p_i^*\rVert \leq \lvert\hat{\lambda}_i - \lambda_i^*\rvert$ and $\lVert\hat{\q}_i - \q_i^*\rVert \leq \lvert\hat{\mu}_i - \mu_i^*\rvert$.
\end{lemma}
\begin{proof}
  If $\hat{\lambda}_i \geq \lambda_i^*$, then the inequality is clear. If $\hat{\lambda}_i < \lambda_i^*$, then $\hat{\p}_i = \hat{\lambda}_i\p_i^*/\lambda_i^*$, so
  $\lVert\hat{\p}_i - \p_i^*\rVert = \lVert(\hat{\lambda}_i - \lambda_i^*)\p_i^*/\lambda_i^*\rVert \leq \lvert\hat{\lambda}_i - \lambda_i^*\rvert$, since $\lVert \p_i^*\rVert \leq \lambda_i^*$.
  The proof of $\lVert\hat{\q}_i - \q_i^*\rVert \leq \lvert\hat{\mu}_i - \mu_i^*\rvert$ is analogous.
\end{proof}

\begin{theorem} \label{thm:slab-width}
  Define $\delta_{\mA}$ to be the distance between the two convex hulls
  \[
    \delta_{\mA} = \dist\left(\conv\left(\bigcup_{i=1}^j C_i\right), \conv\left(\bigcup_{i=1}^l D_i\right)\right).
  \]
  Then $\delta_{\mA} \leq 3\lVert\v_D\rVert$.
\end{theorem}
\begin{proof}
  Since $\b - \mA\x^* = \v_D = (s^*,t^*,\w^*)$, we have that $\b - \mA\hat{\x} = (0,0,\r)$, where
  \[
    \r := \w^* + \sum_{i=1}^j\big((\lambda_i^* - \hat{\lambda}_i)\c_i + A_i(\p_i^* - \hat{\p}_i)\big) + \sum_{i=1}^l\big(-(\mu_i^* - \hat{\mu}_i)\d_i + B_i(\q_i^* - \hat{\q}_i)\big).
  \]
  From Lemma~\ref{lem:simplex-proj}, 
  \[
    \sum_{i=1}^j \lvert\lambda_i^* - \hat{\lambda}_i\rvert \leq \lvert s^*\rvert,\quad \sum_{i=1}^l \lvert\mu_i^* - \hat{\mu}_i\rvert \leq \lvert t^*\rvert,
  \]
  so from Lemma~\ref{lem:pq-diff}, we have that
  \[
    \sum_{i=1}^j \lVert\p_i^* - \hat{\p}_i\rVert \leq \lvert s^*\rvert,\quad \sum_{i=1}^l \lVert\q_i^* - \hat{\q}_i\rVert \leq \lvert t^*\rvert.
  \]
  Also by the normalization assumption \eqref{ass}, $\lVert\c_i\rVert \leq 1$, $\lVert A_i\rVert \leq 1$, $\lVert\d_i\rVert \leq 1$, $\lVert B_i\rVert \leq 1$.
  These inequalities allow us to conclude that
  \[
    \lVert\r\rVert \leq \lVert\w^*\rVert + 2\lvert s^*\rvert + 2\lvert t^*\rvert \leq 3\lVert\v_D\rVert.
  \]
    Since $\hat{\x}\in\mC$, $\|\hat{\p}_i\| \leq \hat{\lambda}_i$. Define for all $i \in [j]$ with $\hat{\lambda}_i > 0$
    \[
        \z_i := \c_i + \frac{1}{\hat{\lambda}_i}A_i\hat{\p}_i.
    \]
    Then
    \[
        \|A_i^{-1}(\z_i - \c_i)\| \leq 1,
    \]
    so $\z_i \in C_i$. Hence, we can define
    \[
        \u:=\sum_{i=1}^j(\hat{\lambda}_i\c_i+A_i\hat{\p}_i),
    \]
    such that $\u = \sum_{i : \hat{\lambda}_i > 0} \hat{\lambda}_i z_i$. Since $\sum_{i=1}^j\hat{\lambda}_i = 1$ and $\hat{\lambda}_i \geq 0$, it follows that
    \[
        \u \in \conv\left(\bigcup_{i=1}^j C_i\right).
    \]
    A similar construction shows that
    \[
        \v:=\sum_{i=1}^l
        (\hat{\mu}_i\d_i-B_i\hat{\q}_i)
        \in\conv\left(\bigcup_{i=1}^l D_i\right).
    \]
    By the definition of $\r$, we have $\r = \v - \u$. Consequently,
    \[
        \delta_{\mA} \leq\|\u-\v\| = \|\r\| \leq 3\|\v_D\|.
    \]
\end{proof}

\begin{corollary} \label{cor:vk-vd-data-bound}
  Suppose the two collections of ellipsoids are strictly separable. The residual $\v_k = \b - \mA\x_k = (s_k,t_k,\w_k)$ gives rise to the strict separator $H_k$, provided
  \[
    \lVert\v_k - \v_D\rVert < \frac{\delta_{\mA}^2}{27\sqrt{2}}.
  \]
\end{corollary}
\begin{proof}
  We have shown in Theorem~\ref{thm:vk-vd-approx-sep-bound} that $\v_k$ is a strict separator provided $\lVert\v_k - \v_D\rVert < \lVert\v_D\rVert^2/\sqrt{18}$.
  We have also shown in Theorem~\ref{thm:slab-width} that we can lower bound $\lVert\v_D\rVert$ by the distance between the two convex hulls $\delta_{\mA}$: $\lVert\v_D\rVert \geq \delta_{\mA}/3$.
  Hence, to ensure $\lVert\v_k - \v_D\rVert < \lVert\v_D\rVert^2/\sqrt{18}$, we require that
  \[
    \lVert\v_k - \v_D\rVert < \frac{\delta_{\mA}^2}{27\sqrt{2}}.
  \]
\end{proof}
\begin{theorem}
  Suppose the two collections of ellipsoids are strictly separable. FISTA applied to the reformulation of the ESP \eqref{prob:socp-esp-dual2} with $\x_0 = \bz$ finds a separating hyperplane in at most
  \[
    \Bigg\lceil\frac{108\sqrt{6(j+l)}}{\delta_{\mA}^2}\Bigg\rceil
  \]
  iterations.
\end{theorem}
\begin{proof}
  From Corollary~\ref{cor:vk-vd-data-bound}, we know that $\v_k = (s_k,t_k,\w_k)$ encodes a separating hyperplane when $\lVert\v_k - \v_D\rVert < \delta_{\mA}^2/(27\sqrt{2})$. We want to know
  when $\v_k$ lies in an open ball of radius $\delta_{\mA}^2/(27\sqrt{2})$ around $\v_D$. Recall (cf. \cite[Theorem 10.34]{Beck2017}) that for iterates
  $\x_k$ generated by FISTA and any $\bar{\x} \in \{\x \in \mC : \b - \mA\x = \v_D\}$,
  \[
    f(\x_k) - f(\bar{\x}) \leq \frac{2L\lVert\x_0 - \bar{\x}\rVert^2}{(k+1)^2},
  \]
  where $\x_0$ is the starting point of the algorithm and $L = \|\mA\|^2$ is the Lipschitz constant for $\nabla f$. We thus obtain
  \[
    f(\x_k) - f(\bar{\x}) \leq \frac{2\lVert\mA\rVert^2\lVert\x_0 - \bar{\x}\rVert^2}{(k+1)^2}.
  \]
  From the proof of Theorem~\ref{thm:pdhg-to-fista}, we have $\lVert\v_k - \v_D\rVert \leq \sqrt{2}\sqrt{f(\x_k) - f(\bar{\x})}$, so applying this to the above we obtain the bound
  \[
    \lVert\v_k - \v_D\rVert \leq \frac{2\lVert\mA\rVert\lVert\x_0 - \bar{\x}\rVert}{k+1}.
  \]
  For any $\x \in \mC$,
  \begin{align*}
    \|\x\|^2 &= \sum_{i=1}^j\left(\lambda_i^2 + \|\p_i\|^2\right) + \sum_{i=1}^l\left(\mu_i^2 + \|\q_i\|^2\right)\\
                          &\leq 2\sum_{i=1}^j \lambda_i^2 + 2\sum_{i=1}^l \mu_i^2\\
                          &\leq 2\left(\left(\sum_{i=1}^j \lambda_i\right)^2 + \left(\sum_{i=1}^l \mu_i\right)^2\right)\\
                          &\leq 2\|\mA\x\|^2.
  \end{align*}
  From \eqref{eqn:vd-proj}, $\mA\bar{\x} = \b - \v_D = \b - \proj_{\b - \mA(\mC)}(\bz) = \proj_{\mA(\mC)}(\b)$, and $\bz \in \mA(\mC)$, so by the nonexpansiveness of the projection
  \[
    \|\x_0 - \bar{\x}\| \leq \sqrt{2}\|\mA\bar{\x}\| \leq \sqrt{2}\|\proj_{\mA(\mC)}(\b) - \proj_{\mA(\mC)}(\bz)\| \leq \sqrt{2}\|\b\| = 2.
  \]
  It follows that we can guarantee that $\v_k$ lies in an open ball of radius $\delta_{\mA}^2/27\sqrt{2}$ around $\v_D$ when
  $4\lVert\mA\rVert/(k+1) < \delta_{\mA}^2/(27\sqrt{2})$. Under Assumption~\ref{ass}, a sufficient condition is
  \begin{equation*}
    k \geq \frac{108\sqrt{6(j+l)}}{\delta_{\mA}^2} \geq \frac{108\sqrt{2}\lVert\mA\rVert}{\delta_{\mA}^2}.
  \end{equation*}
\end{proof}

\paragraph{Verifying separation.}
Suppose the collections of ellipsoids are separable. For the FISTA residual $\v_k = (s_k,t_k,\w_k)$, the hyperplane $H_k$ strictly separates the collections of ellipsoids if and only if
\[
    \begin{aligned}
        \mathbf{c}_i^T\mathbf{w}_k+\lVert A_i^T\mathbf{w}_k\rVert & < m_k, && i=1,\ldots,j,\\
        \mathbf{d}_i^T\mathbf{w}_k-\lVert B_i^T\mathbf{w}_k\rVert & > m_k, && i=1,\ldots,l,
    \end{aligned}
\]
where the expressions on the left hand side are the maximum and minimum, respectively, of $\w_k^T\z$ over the corresponding ellipsoids (see, for example, \cite[Lemma 6.18]{Jiang2023}).
It follows that separation can be verified efficiently at each FISTA iteration without the knowledge of $\v_D$.

\section{Improved complexity bound} \label{sec:improved}
We now derive the improved complexity bound of $\O(\delta_{\mA}^{-1})$ for ellipsoid separation. The proof of the improved bound
uses the weighted average of the residuals
\[
  \bar{\v}_n := \frac{1}{\theta_n^2}\sum_{k=0}^n \theta_k(\b - \mA\y_k),
\]
where $\theta_k$ are the momentum parameters from FISTA, computed recursively as
\[
  \theta_{k+1} := \frac{1 + \sqrt{1+4\theta_k^2}}{2},\quad \theta_0 = 1.
\]
We first state some useful properties of the momentum parameters $\theta_k$.
\begin{proposition} \label{prop:tk}
  For every $k \ge 0$,

  \vspace{0.5em}
  \par\smallskip\noindent
  \(
  \begin{aligned}
    \text{\normalfont(i)} \quad
      & \theta_{k+1}^{2} = \theta_k^{2} + \theta_{k+1}, \\[5pt]
    \text{\normalfont(ii)} \quad
      & \sum_{i=0}^{k} \theta_i = \theta_k^{2}, \\[5pt]
    \text{\normalfont(iii)} \quad
      & \theta_k \ge \frac{k+2}{2}.
  \end{aligned}
  \)
\end{proposition}
\begin{proof}
  (i) can be verified using the definition of $\theta_k$. For (ii), we may use (i) and that $\theta_0 = 1$ to write the telescoping sum:
  \[
    \sum_{i=0}^k \theta_i = \theta_0 + \sum_{i=1}^k (\theta_i^2 - \theta_{i-1}^2) = \theta_0 + \theta_k^2 - \theta_0^2 = \theta_k^2.
  \]
  (iii) is Lemma~10.33 of \cite{Beck2017}.
\end{proof}
Using the identity $\theta_{k+1}^2 = \theta_k^2 + \theta_{k+1}$ from Proposition~\ref{prop:tk}(i), we can compute $\bar{\v}_k$ ``online" throughout the FISTA algorithm using the recursion
\[
  \bar{\v}_{k+1} = \left(1 - \frac{1}{\theta_{k+1}}\right)\bar{\v}_k + \frac{1}{\theta_{k+1}}\r_{k+1},\quad \bar{\v}_0 = \r_0,
\]
where $\r_k = \b - \mA\y_k$. Thus, we consider the following modified FISTA algorithm:
\begin{align*}
  \x_{k+1} &= \proj_{\mC}\left(\y_k + \frac{1}{\|\mA\|^2}\mA^T\r_k\right)\\
  \theta_{k+1} &= \frac{1 + \sqrt{1 + 4\theta_k^2}}{2}\\
  \y_{k+1} &= \x_{k+1} + \left(\frac{\theta_k - 1}{\theta_{k+1}}\right)\left(\x_{k+1} - \x_k\right)\\
  \bar{\v}_{k+1} &= \left(1 - \frac{1}{\theta_{k+1}}\right)\bar{\v}_k + \frac{1}{\theta_{k+1}}\left(\b - \mA\y_{k+1}\right),
\end{align*}
where $\x_0 = \y_0 = \bz$, $\theta_0 = 1$, and we have taken the smoothness parameter to be $L = \|\mA\|^2$.

Define the convex hulls of the ellipsoids as follows:
\[
  K := \conv(C_1 \cup \hdots \cup C_j),\quad H := \conv(D_1 \cup \hdots \cup D_l).
\]

\begin{lemma} \label{lem:Au}
  For every $\p \in K, \q \in H$ there exists $\u_p, \u_q \in \mC$ with $\|\u_p\|^2 \leq 2$ and $\|\u_q\|^2 \leq 2$, and
  \[
    \mA\u_p = \begin{pmatrix}
      1\\0\\ \p
    \end{pmatrix},\quad \mA\u_q = \begin{pmatrix}
      0\\1\\ -\q
    \end{pmatrix}.
  \]
  % Hence, we may define $\u := \u_p + \u_q$ such that $\|\u\| \leq 2$ and
  % \[
  %   \mA\u = \begin{pmatrix}
  %     1\\1\\ \p - \q
  %   \end{pmatrix}.
  % \]
\end{lemma}
\begin{proof}
  Since $\p \in K$, there exists $\lambda_i \geq 0$ and $\z_i \in \R^d$ with $\sum_{i=1}^j \lambda_i = 1$ and $\|\z_i\| \leq 1$
  such that
  \[
    \p = \sum_{i=1}^j \lambda_i \left(\c_i + A_i\z_i\right).
  \]
  Define the vector $\u_p$ so that its first $j(d+1)$ entries are given by blocks of size $d+1$
  containing $(\lambda_i, \lambda_i\z_i)$ for each $i \in [j]$, while the remaining $l(d+1)$ entries are zero. Hence,
  $\u_p$ looks like
  \[
    \u_p = \begin{pmatrix}
      \lambda_1\\
      \lambda_1\z_1\\
      \lambda_2\\
      \lambda_2\z_2\\
      \vdots\\
      \lambda_j\\
      \lambda_j\z_j\\
      0\\
      \bz\\
      \vdots\\
      0\\
      \bz
    \end{pmatrix}
  \]
  Since $\|\z_i\| \leq 1$, it must be that $\|\lambda_i\z_i\| \leq \lambda_i$ and $\u_p \in \mC$. Similarly, $\q \in H$, so
  there exists $\mu_i \geq 0$ and $\z_{j+i} \in \R^d$ with $\sum_{i=1}^l \mu_i = 1$ and $\|\z_{j+i}\| \leq 1$
  such that
  \[
    \q = \sum_{i=1}^l \mu_i \left(\d_i - B_i\z_{j+i}\right).
  \]
  Define the vector $\u_q$ in the same way as $\u_p$ but so that the first $j(d+1)$ entries are zero and the last $l(d+1)$ entries are given by
  $(\mu_i, \mu_i\z_{j+i})$ for each $i \in [l]$. It must be that $\u_q \in \mC$ by the same reasoning as for $\u_p$.
  To finish, we have
  \begin{align*}
    \|\u_p\|^2 &= \sum_{i=1}^j\left(\lambda_i^2 + \|\lambda_i\z_i\|^2\right) \leq 2\sum_{i=1}^{j}\lambda_i^2 \leq 2\left(\sum_{i=1}^{j}\lambda_i\right)^2 = 2,\\
    \|\u_q\|^2 &= \sum_{i=1}^l\left(\mu_i^2 + \|\mu_i\z_{j+i}\|^2\right) \leq 2\sum_{i=1}^{l}\mu_i^2 \leq 2\left(\sum_{i=1}^{l}\mu_i\right)^2 = 2.\\
  \end{align*}
  % Now define $\u := \u_p + \u_q$. It follows that
  % \[
  %   \mA\u = \mA\u_p + \mA\u_q\\
  %         = \begin{pmatrix}
  %           1\\
  %           0\\
  %           \p
  %         \end{pmatrix} + 
  %         \begin{pmatrix}
  %           0\\
  %           1\\
  %           -\q
  %         \end{pmatrix}
  %         = \begin{pmatrix}
  %           1\\
  %           1\\
  %           \p - \q
  %         \end{pmatrix}.
  % \]
  % Moreover,
  % \begin{align*}
  %   \|\u\|^2 &= \|\u_p\|^2 + \|\u_q\|^2\\
  %            &= \sum_{i=1}^j\left(\lambda_i^2 + \|\lambda_i\z_i\|^2\right) + \sum_{i=1}^l\left(\mu_i^2 + \|\mu_i\z_{j+i}\|^2\right)\\
  %            &\leq 2\sum_{i=1}^{j}\lambda_i^2 + 2\sum_{i=1}^{l}\mu_i^2\\
  %            &\leq 2\left(\sum_{i=1}^{j}\lambda_i\right)^2 + 2\left(\sum_{i=1}^{l}\mu_i\right)^2\\
  %            &\leq 4.
  % \end{align*}
  % Taking the square root gives the desired bound.
\end{proof}

% We will need the following bound from our paper.
% \begin{lemma}[{\cite[Theorem~3]{Hough2026}}] \label{lem:vd-bound}
%   Define $\delta_{\mA}$ to be the distance between the two convex hulls
%   \[
%     \delta_{\mA} = \dist\left(\conv\left(\bigcup_{i=1}^j C_i\right), \conv\left(\bigcup_{i=1}^l D_i\right)\right).
%   \]
%   Then $\delta_{\mA} \leq 3\lVert\v_D\rVert$.
% \end{lemma}

In Section~\ref{sec:esp-fista} we used the raw residuals $\v_k = (s_k,t_k,\w_k)$ to define the quantity $m_k = (-s_k + t_k)/2$, and stopped
when $\max_{\p \in K} \w_k^T\p < m_k < \min_{\q \in H} \w_k^T\q$. We now use the averaged residuals $\bar{\v}_k = (\bar{s}_k, \bar{t}_k, \bar{\w}_k)$, so we define
\[
  h_k := \frac{-\bar{s}_k + \bar{t}_k}{2},
\]
and our stopping test is analogous to the one we used in Section~\ref{sec:esp-data-dep-bound}: $U_k < h_k < V_k$, where
\begin{align*}
  &U_k := \max_{\p \in K} \bar{\w}_k^T\p = \max_{i \in [j]} \left\{\c_i^T\bar{\w}_k + \|A_i^T\bar{\w}_k\|\right\}\\
  &V_k := \min_{\q \in H} \bar{\w}_k^T\q = \min_{i \in [l]} \left\{\d_i^T\bar{\w}_k - \|B_i^T\bar{\w}_k\|\right\}.
\end{align*}
To determine the number of iterations to find a separating hyperplane, we wish to determine the first $k$ for which
\[
  \min\{h_k - U_k, V_k - h_k\} > 0.
\]
We do this by constructing a lower bound on the left hand side above and determining the first iteration for which it becomes positive.

\begin{lemma} \label{lem:thm-ub}
  For any $\u \in \mC$, we have for all $k \geq 1$,
  \[
    \bar{\v}_k^T\mA\u \leq \frac{\|\mA\|^2}{\theta_k^2}\u^T\z_{k+1},
  \]
  where $\z_i$ are defined by $\z_{i+1} = (1-\theta_i)\x_i + \theta_i\x_{i+1}$, $\z_0 = \bz$
\end{lemma}
\begin{proof}
  By the first-order optimality conditions of the FISTA projection step, for all $i = 0,\hdots,k$
  \[
    \n_i := \mA^T\r_i - \|\mA\|^2(\x_{i+1} - \y_i) \in \mN_{\mC}(\x_{i+1}).
  \]
  Since $\n_i \in \mN_{\mC}(\x_{i+1})$ and $\mC$ is a convex cone, for any $\u \in \mC$ we have
  $\x_{i+1} + \u \in \mC$. It follows from the definition of the normal cone that
  $\n_i^T\u \leq 0$ for all $\u \in \mC$, so $\n_i \in \mC^{\circ}$.

  Rearranging and multiplying the expression for $\n_i$ above by $\theta_i$ gives
  \begin{equation} \label{eqn:tAr}
    \theta_i\mA^T\r_i = \theta_i\n_i + \|\mA\|^2\theta_i(\x_{i+1} - \y_i) = \theta_i\n_i + \|\mA\|^2(\z_{i+1} - \z_i).
  \end{equation}
  Summing from $i=0$ to $k$
  makes the $\z_{i+1} - \z_i$ terms telescope:
  \[
    \sum_{i=0}^k \theta_i \mA^T\r_i = \sum_{i=0}^k \theta_i\n_i + \|\mA\|^2(\z_{k+1} - \z_0).
  \]
  Dividing both sides by $\theta_{k}^2$ and using $\z_0 = \bz$ gives
  \[
    \mA^T\bar{\v}_k = \frac{1}{\theta_k^2}\sum_{i=0}^k \theta_i\n_i + \frac{\|\mA\|^2}{\theta_k^2}\z_{k+1}.
  \]
  Taking the inner product with any $\u \in \mC$ and using that $\n_i \in \mC^{\circ}$, we get
  \[
    \bar{\v}_k^T\mA\u \leq \frac{\|\mA\|^2}{\theta_k^2}\u^T\z_{k+1}.
  \]
\end{proof}

\begin{lemma} \label{lem:thm-lb}
  Suppose $\z_i$ are defined as in Lemma~\ref{lem:thm-ub}. Then for all $k \geq 1$,
  \[
    % \b^T\bar{\v}_k \geq \frac{\|\v_D\|^2}{2} + \frac{L}{2t_k^2}\|\z_{k+1}\|^2.
    \b^T\bar{\v}_k \geq \|\v_D\|^2 + \frac{\|\mA\|^2}{2\theta_k^2}\|\z_{k+1}\|^2.
  \]
\end{lemma}
\begin{proof}
  Recall the definition of $\n_i$ from Lemma~\ref{lem:thm-ub}:
  \[
    \n_i = \mA^T\r_i - \|\mA\|^2(\x_{i+1} - \y_i) \in \mN_{\mC}(\x_{i+1}).
  \]
  Since $\bz \in \mC$ and $2\x_{i+1} \in \mC$ along with
  $\n_i \in \mN_{\mC}(\x_{i+1})$ for each $i$, $\n_i^T\x_{i+1} = 0$.
  Moreover, since $\x_i \in \mC$ and $\n_i \in \mC^{\circ}$, $\n_i^T\x_i \leq 0$. Thus,
  \[
    \n_i^T\z_{i+1} = (1-\theta_i)\n_i^T\x_i + \theta_i\n_i^T\x_{i+1} = (1-\theta_i)\n_i^T\x_i \geq 0.
  \]
  Now, using \eqref{eqn:tAr} from Lemma~\ref{lem:thm-ub} we obtain
  \begin{align}
    \theta_i\r_i^T\mA\z_{i+1} &= \theta_i\n_i^T\z_{i+1} + \|\mA\|^2(\z_{i+1} - \z_i)^T\z_{i+1}\notag\\
                         &\geq \|\mA\|^2(\z_{i+1} - \z_i)^T\z_{i+1}\notag\\
                         &= \frac{\|\mA\|^2}{2}\left(\|\z_{i+1}\|^2 - \|\z_i\|^2 + \|\z_{i+1} - \z_i\|^2\right).\label{eqn:trA-lb}
  \end{align}
  Writing $\x = \y_i + (\x - \y_i)$ for any iteration $i$, we may expand the squared norm in the objective $f(\x) = \frac{1}{2}\|\mA\x - \b\|^2$ to obtain
  \[
    f(\x) = \frac{1}{2}\|\r_i - \mA(\x - \y_i)\|^2 = \frac{1}{2}\|\r_i\|^2 - \r_i^T\mA(\x - \y_i) + \frac{1}{2}\|\mA(\x - \y_i)\|^2.
  \]
  Applying the above to $\x = \x_{i+1}$ and $\x = \x_i$, and dropping the nonnegative term for the second, gives
  \begin{align}
    f(\x_{i+1}) &\leq \frac{1}{2}\|\r_i\|^2 - \r_i^T\mA(\x_{i+1} - \y_i) + \frac{\|\mA\|^2}{2}\|\x_{i+1} - \y_i\|^2,\label{eqn:fxip1-ub}\\
    f(\x_i) &\geq \frac{1}{2}\|\r_i\|^2 - \r_i^T\mA(\x_i - \y_i).\label{eqn:fxi-lb}
  \end{align}
  Note $\theta_i^2 - \theta_i \geq 0$. Multiplying \eqref{eqn:fxip1-ub} by $\theta_i^2$ and subtracting $\theta_i^2 - \theta_i$ multiplied by \eqref{eqn:fxi-lb} gives
  \begin{align}
    \theta_i^2f(\x_{i+1}) - (\theta_i^2 - \theta_i)f(\x_i) &\leq \frac{\theta_i}{2}\|\r_i\|^2 - \r_i^T\mA\bigg(\theta_i^2(\x_{i+1} - \y_i) - (\theta_i^2 - \theta_i)(\x_i - \y_i)\bigg)\notag\\
                                            &\ \ \ \ + \frac{\|\mA\|^2}{2}\theta_i(\x_{i+1} - \y_i)^T\big(\theta_i(\x_{i+1} - \y_i)\big)\notag\\
                                            &= \frac{\theta_i}{2}\|\r_i\|^2 - \theta_i\r_i^T\mA(\z_{i+1} - \y_i) + \frac{\|\mA\|^2}{2}\|\z_{i+1} - \z_i\|^2\notag\\
                                            &= \frac{\theta_i}{2}\|\r_i\|^2 + \theta_i\r_i^T(\b - \r_i) - \theta_i\r_i^T\mA\z_{i+1} + \frac{\|\mA\|^2}{2}\|\z_{i+1} - \z_i\|^2\notag\\
                                            &= \theta_i\bigg(\b^T\r_i - \frac{1}{2}\|\r_i\|^2\bigg) - \theta_i\r_i^T\mA\z_{i+1} + \frac{\|\mA\|^2}{2}\|\z_{i+1} - \z_i\|^2\label{eqn:t2fdiff},
  \end{align}
  where in the first equality we used that
  \[
    \theta_i^2(\x_{i+1} - \y_i) - (\theta_i^2 - \theta_i)(\x_i - \y_i) = \theta_i(\theta_i\x_{i+1} + (1-\theta_i)\x_i - \y_i) = \theta_i(\z_{i+1} - \y_i),
  \]
  and $\z_{i+1} - \z_i = \theta_i(\x_{i+1} - \y_i)$, and in the second-last equality we used that $\mA\y_i = \b - \r_i$.
  Substituting \eqref{eqn:trA-lb} into \eqref{eqn:t2fdiff} gives for all $i \geq 0$
  \begin{align*}
    \theta_i^2f(\x_{i+1}) - (\theta_i^2 - \theta_i)f(\x_i) &\leq \theta_i\bigg(\b^T\r_i - \frac{1}{2}\|\r_i\|^2\bigg)\\
                                            &\ \ \ \ - \frac{\|\mA\|^2}{2}\left(\|\z_{i+1}\|^2 - \|\z_i\|^2 + \|\z_{i+1} - \z_i\|^2\right)\\
                                            &\ \ \ \ + \frac{\|\mA\|^2}{2}\|\z_{i+1} - \z_i\|^2\\
                                            &= \theta_i\bigg(\b^T\r_i - \frac{1}{2}\|\r_i\|^2\bigg) - \frac{\|\mA\|^2}{2}\left(\|\z_{i+1}\|^2 - \|\z_i\|^2\right).
  \end{align*}
  To complete the proof, define the potential
  \[
    E_{i+1} = \theta_i^2f(\x_{i+1}) + \frac{\|\mA\|^2}{2}\|\z_{i+1}\|^2,\quad E_0 = 0,
  \]
  and use $\theta_i^2 - \theta_i = \theta^2_{i-1}$ for $i \geq 1$ to write for all $i \geq 1$,
  \begin{align*}
    E_{i+1} - E_i &= \theta_i^2f(\x_{i+1}) - \theta_{i-1}^2f(\x_{i})  + \frac{\|\mA\|^2}{2}\left(\|\z_{i+1}\|^2 - \|\z_{i}\|^2\right)\\
                  &= \theta_i^2f(\x_{i+1}) - (\theta_i^2 - \theta_i)f(\x_i) + \frac{\|\mA\|^2}{2}\left(\|\z_{i+1}\|^2 - \|\z_{i}\|^2\right)\\
                  &\leq \theta_i\bigg(\b^T\r_i - \frac{1}{2}\|\r_i\|^2\bigg).
  \end{align*}
  Since $\theta_0^2 - \theta_0 = 1 - 1 = 0$ and $\z_0 = \bz$, the bound also works for the case where $i = 0$:
  \begin{align*}
    E_{1} - E_0 &= \theta_0^2f(\x_{1}) + \frac{\|\mA\|^2}{2}\|\z_{1}\|^2\\
                  &= \theta_0^2f(\x_{1}) - (\theta_0^2 - \theta_0)f(\x_0) + \frac{\|\mA\|^2}{2}\left(\|\z_{1}\|^2 - \|\z_{0}\|^2\right)\\
                  &\leq \theta_0\bigg(\b^T\r_0 - \frac{1}{2}\|\r_0\|^2\bigg).
  \end{align*}
  Summing the above inequality from $i=0$ to $i = k$ gives left hand side
  \[
    \sum_{i=0}^k \big(E_{i+1} - E_i\big) = E_{k+1} - E_0 = \theta_k^2f(\x_{k+1}) + \frac{\|\mA\|^2}{2}\|\z_{k+1}\|^2,
  \]
  and right hand side
  \begin{align*}
    \sum_{i=0}^k\theta_i\bigg(\b^T\r_i - \frac{1}{2}\|\r_i\|^2\bigg) &= \b^T\left(\sum_{i=0}^k\theta_i \r_i\right) - \frac{1}{2}\sum_{i=0}^k \theta_i \|\r_i\|^2\\
                                                                &= \theta_k^2\b^T\bar{\v}_k - \frac{1}{2}\sum_{i=0}^k \theta_i \|\r_i\|^2.
  \end{align*}
  Rearranging and dividing by $\theta_k^2$ gives
  % \begin{align*}
  %   \b^T\bar{\v}_k &\geq f(\x_{k+1}) + \frac{L}{2t_k^2}\|\z_{k+1}\|^2 + \frac{1}{2t_k^2}\sum_{i=0}^k \theta_i \|\r_i\|^2\\
  %                  &\geq f(\x_{k+1}) + \frac{L}{2t_k^2}\|\z_{k+1}\|^2\\
  %                  &\geq \frac{\|\v_D\|^2}{2} + \frac{L}{2t_k^2}\|\z_{k+1}\|^2,
  % \end{align*}
  \begin{equation} \label{eqn:bv-eq}
    \b^T\bar{\v}_k \geq f(\x_{k+1}) + \frac{\|\mA\|^2}{2\theta_k^2}\|\z_{k+1}\|^2 + \frac{1}{2\theta_k^2}\sum_{i=0}^k \theta_i \|\r_i\|^2.
  \end{equation}
  Using the update rule $\y_i = \x_i + \frac{\theta_{i-1} - 1}{\theta_i}(\x_i - \x_{i-1})$, we may write
  \begin{align*}
    \sum_{i=0}^k \theta_i \y_i &= \x_0 + \sum_{i=1}^{k}\theta_i \y_i\\
                         &= \x_0 + \sum_{i=1}^{k}\bigg((\theta_i + \theta_{i-1} - 1)\x_i - (\theta_{i-1} - 1)\x_{i-1}\bigg)\\
                         &= \x_0 + \sum_{i=1}^{k}(\theta_i + \theta_{i-1} - 1)\x_i - \sum_{i=0}^{k-1}(\theta_{i} - 1)\x_{i}\\
                         &= \x_0 + \sum_{i=1}^{k}(\theta_i + \theta_{i-1} - 1)\x_i - \sum_{i=1}^{k-1}(\theta_{i} - 1)\x_{i},
  \end{align*}
  where on the second-last line we reindexed the second summation and on the last line we used $\theta_0 - 1 = 0$. Equating coefficients gives
  \[
    \sum_{i=0}^k \theta_i \y_i = \x_0 + \sum_{i=1}^{k-1}\theta_{i-1}\x_{i} + (\theta_k + \theta_{k-1} - 1)\x_k \in \mC.
  \]
  Applying Jensen's inequality to $\y \mapsto \|\b - \mA\y\|^2$ and using that $\sum_{i=0}^k \theta_i \y_i \in \mC$ gives
  \[
    \frac{1}{\theta_k^2}\sum_{i=0}^k \theta_i\|\r_i\|^2 \geq \Bigg\lVert\b - \mA\left(\frac{1}{\theta_k^2}\sum_{i=0}^k \theta_i\y_i\right)\Bigg\rVert^2 \geq \|\v_D\|^2.
  \]
  Now we may return to \eqref{eqn:bv-eq} to obtain the bound
  \[
    \b^T\bar{\v}_k \geq \|\v_D\|^2 + \frac{\|\mA\|^2}{2\theta_k^2}\|\z_{k+1}\|^2,
  \]
  since $f(\x_{k+1}) \geq f(\x^*) = \frac{1}{2}\|\v_D\|^2$.
\end{proof}

\begin{theorem}
  The algorithm stops with the hyperplane $\{\bar{\w}_k^T\z = h_k\}$ separating the ellipsoids whenever
  \[
    k + 2 > \frac{12\sqrt{3(j+l)}}{\delta_{\mA}}.
  \]
\end{theorem}
\begin{proof}
  From Lemma~\ref{lem:Au}, $\mA\u_p = (1,0,\p)$ and $\mA\u_q = (0,1,-\q)$. This allows us to write for any $\p \in K$ and $\q \in H$,
  \begin{align*}
    h_k - \bar{\w}_k^T\p &= \frac{1}{2}\begin{pmatrix}-1\\ 1\\ \bz\end{pmatrix}^T\bar{\v}_k - \begin{pmatrix}0\\0\\ \p\end{pmatrix}^T\bar{\v}_k\\
                         &= \begin{pmatrix}-1\\0\\-\p\end{pmatrix}^T\bar{\v}_k + \frac{1}{2}\begin{pmatrix}1\\1\\\bz\end{pmatrix}^T\bar{\v}_k\\
                         &= \frac{1}{2}\b^T\bar{\v}_k - \bar{\v}_k^T\mA\u_p,
  \end{align*}
  and similarly for the other gap
  \begin{align*}
    \bar{\w}_k^T\q - h_k &= \begin{pmatrix}0\\ 0\\ \q\end{pmatrix}^T\bar{\v}_k - \frac{1}{2}\begin{pmatrix}-1\\1\\ \bz\end{pmatrix}^T\bar{\v}_k\\
                         &= \frac{1}{2}\begin{pmatrix}1\\1\\\bz\end{pmatrix}^T\bar{\v}_k - \begin{pmatrix}0\\1\\-\q\end{pmatrix}^T\bar{\v}_k\\
                         &= \frac{1}{2}\b^T\bar{\v}_k - \bar{\v}_k^T\mA\u_q.
  \end{align*}
  We want a lower bound on the above gaps. Lemma~\ref{lem:thm-ub} provides an upper bound on both $\bar{\v}_k^T\mA\u_q$ and $\bar{\v}_k^T\mA\u_p$,
  while Lemma~\ref{lem:thm-lb} provides a lower bound on $\b^T\bar{\v}_k$. Consider $\u \in \{\u_p, \u_q\} \subseteq \mC$ where $\u_p$ and $\u_q$ are defined
  in Lemma~\ref{lem:Au}, and use Lemma~\ref{lem:thm-ub} and Lemma~\ref{lem:thm-lb} to write
  \begin{align*}
    \frac{1}{2}\b^T\bar{\v}_k - \bar{\v}_k^T\mA\u &\geq \frac{\|\v_D\|^2}{2} + \frac{\|\mA\|^2}{4\theta_k^2}\|\z_{k+1}\|^2 - \frac{\|\mA\|^2}{\theta_k^2}\u^T\z_{k+1}\\
                         &= \frac{\|\v_D\|^2}{2} + \frac{\|\mA\|^2}{4\theta_k^2}\big(\|\z_{k+1}\|^2 - 4\u^T\z_{k+1}\big)\\
                         &= \frac{\|\v_D\|^2}{2} + \frac{\|\mA\|^2}{4\theta_k^2}\big(\|\z_{k+1} - 2\u\|^2 - 4\|\u\|^2\big)\\
                         &\geq \frac{\|\v_D\|^2}{2} - \frac{2\|\mA\|^2}{\theta_k^2},
  \end{align*}
  where the last inequality comes from $\|\z_{k+1} - 2\u\|^2 \geq 0$ and $\|\u_p\|^2, \|\u_q\|^2 \leq 2$ as shown in Lemma~\ref{lem:Au}.

  From Theorem~\ref{thm:slab-width}, we have $\|\v_D\| \geq \delta_{\mA}/3$, while Proposition~\ref{prop:tk}(iii) gives $\theta_k \geq (k+2)/2$. It follows from the above that
  \begin{align*}
    \min\{h_k - \bar{\w}_k^T\p, \bar{\w}_k^T\q - h_k\} &\geq \frac{\|\v_D\|^2}{2} - \frac{2\|\mA\|^2}{\theta_k^2}\\
                                                       &\geq \frac{\delta_{\mA}^2}{18} - \frac{8\|\mA\|^2}{(k+2)^2}.
  \end{align*}
  Taking the infimum over $\p \in K$ and $\q \in H$ gives
  \[
    \min\{h_k - U_k, V_k - h_k\} \geq \frac{\delta_{\mA}^2}{18} - \frac{8\|\mA\|^2}{(k+2)^2}.
  \]
  The right hand side above is positive when $(k+2)^2 > 144\|\mA\|^2/\delta_{\mA}^2$, which gives the iteration bound:
  \[
    k+2 > \frac{12\|\mA\|}{\delta_{\mA}}.
  \]
  Under Assumption~\ref{ass}, $\|\mA\| \leq \sqrt{3(j+l)}$. Hence, a sufficient condition for the algorithm to terminate
  with the desired separating hyperplane is
  \[
    k + 2 > \frac{12\sqrt{3(j+l)}}{\delta_{\mA}}.
  \]
\end{proof}

\section*{AI Declaration}
ChatGPT 6 Astra proposed the analysis in Section~\ref{sec:improved}; the authors simplified and clarified it.  AI was not used for the remainder of the paper.

%\begin{acknowledgements}
%\end{acknowledgements}

% The authors declare that they have no conflicts/competing interests relevant to the content of this article.

\bibliographystyle{spmpsci}
\bibliography{assets/refs}
\end{document}